\documentclass[11pt]{article}
\usepackage[T1]{fontenc}
\usepackage{amsfonts}
\usepackage{amsmath}
\usepackage{amssymb}
\usepackage{amsthm}
\usepackage{bbm}
\usepackage{bm}
\usepackage{mathrsfs}
\usepackage{verbatim}
\usepackage{setspace}
\usepackage{color}
\usepackage{pdfsync}
\usepackage{enumitem}
\usepackage{graphicx}
\usepackage{subfigure}
\usepackage{tikz}
\usetikzlibrary{patterns}

\usepackage[pdfborder={0 0 0}]{hyperref}
\hypersetup{
  urlcolor = black,
  pdfauthor = {Promit Ghosal, Marcel Nutz},
  pdfkeywords = {Entropic Optimal Transport; Sinkhorn's Algorithm; IPFP; Stability},
  pdftitle = {On the Convergence Rate of Sinkhorn's Algorithm},
  pdfsubject = {On the Convergence Rate of Sinkhorn's Algorithm},
  pdfpagemode = UseNone
}

\usepackage[capitalize]{cleveref} %

\theoremstyle{plain}
\newtheorem{theorem}{Theorem}[section]
\newtheorem{proposition}[theorem]{Proposition}
\newtheorem{lemma}[theorem]{Lemma}
\newtheorem{corollary}[theorem]{Corollary}

\theoremstyle{definition}

\newtheorem{remark}[theorem]{Remark}

\newtheorem{example}[theorem]{Example}

\theoremstyle{remark}

\renewenvironment{thebibliography}[1]{%
\begin{oldthebibliography}{#1}%
\setlength{\baselineskip}{1em}
\linespread{.2}
\small
\setlength{\parskip}{0.2ex}%
\setlength{\itemsep}{.15em}%
}%
{%
\end{oldthebibliography}%
}
\newcommand{\eps}{\varepsilon}

\newcommand{\N}{\mathbb{N}}

\newcommand{\R}{\mathbb{R}}

\newcommand{\X}{\mathsf{X}}
\newcommand{\Y}{\mathsf{Y}}

\newcommand{\cC}{\mathcal{C}}

\newcommand{\cP}{\mathcal{P}}

\newcommand{\cW}{\mathcal{W}}

\DeclareMathOperator*{\argmin}{arg\, min}

\newcommand{\qforallq}{\quad\mbox{for all}\quad}
\newcommand{\qforq}{\quad\mbox{for}\quad}

\newcommand{\mykill}[1]{}
\numberwithin{equation}{section}

\begin{document}

\title{\vspace{-3em}On the Convergence Rate of Sinkhorn's Algorithm\thanks{The authors thank Stephan Eckstein and Flavien L\'eger for helpful discussions.}}
\date{\today}
\author{
  Promit Ghosal%
  \thanks{
  Department of Statistics, University of Chicago, promit@uchicago.edu. Research supported by NSF Grant DMS-2153661.}
  \and
  Marcel Nutz%
  \thanks{
  Depts.\ of Statistics and Mathematics, Columbia University, mnutz@columbia.edu. Research supported by NSF Grants DMS-1812661, DMS-2106056, DMS-2407074.}
  }
\maketitle \vspace{-1.2em}

\begin{abstract}
  We study Sinkhorn's algorithm for solving the entropically regularized optimal transport problem. Its iterate~$\pi_{t}$ is shown to satisfy $H(\pi_{t}|\pi_{*})+H(\pi_{*}|\pi_{t})=O(t^{-1})$ where~$H$ denotes relative entropy and~$\pi_{*}$ the optimal coupling. This holds for a large class of cost functions and marginals, including quadratic cost with subgaussian marginals. 
  We also obtain the rate $O(t^{-1})$ for the dual suboptimality and $O(t^{-2})$ for the marginal entropies. More precisely, we derive non-asymptotic bounds, and in contrast to previous results on linear convergence that are limited to bounded costs, our estimates do not deteriorate exponentially with the regularization parameter. We also obtain a stability result for~$\pi_{*}$ as a function of the marginals, quantified in relative entropy.
\end{abstract}

\vspace{.3em}

{\small

}
\vspace{.6em}

\maketitle

\section{Introduction}\label{se:intro}

Let $(\X,\mu)$ and $(\Y,\nu)$ be Polish probability spaces and $c:\X\times\Y\to\R$ a measurable and nonnegative (or suitably integrable) cost function. The entropic optimal transport problem with regularization parameter~$\eps\in(0,\infty)$ is
\begin{equation}\label{eq:epsOT}
 \inf_{\pi\in\Pi(\mu,\nu)} \int_{\X\times\Y} c(x,y) \, \pi(dx,dy) + \eps H(\pi|\mu\otimes\nu),
\end{equation}
where $H(\,\cdot\,|\mu\otimes\nu)$ denotes relative entropy (Kullback--Leibler divergence) with respect to the product measure~$\mu\otimes\nu$ of the marginals,
$$
  H(\pi|\mu\otimes\nu):=\begin{cases}
\int \log \frac{d\pi}{d(\mu\otimes\nu)} \,d\pi, & \pi\ll \mu\otimes\nu,\\
\infty, & \pi\not\ll \mu\otimes\nu,
\end{cases} 
$$
and $\Pi(\mu,\nu)$ the set of all couplings; i.e., probability measures~$\pi$ on~$\X\times\Y$ with marginals~$(\mu,\nu)$.
Entropic optimal transport traces back to  Schr\"odinger's thought experiment~\cite{Schrodinger.31} on the most likely evolution of a particle system and its mathematical formalization as the Schr\"odinger bridge (see~\cite{Follmer.88, Leonard.14} for surveys). 
Nowadays, following~\cite{Cuturi.13}, entropic optimal transport is widely used as an approximation of the (unregularized) Monge--Kantorovich optimal transport problem corresponding to $\eps=0$, especially for computing the Wasserstein distance in high-dimensional applications such as machine learning, statistics, image and language processing (e.g., \cite{AlvarezJaakkola.18, WGAN.17,ChernozhukovEtAl.17,RubnerTomasiGuibas.00}). As a result, the cost function of principal interest  is~$c(x,y)=|x-y|^{p}$ on~$\R^{d}\times\R^{d}$, especially for $p=2$, and the convergence properties as $\eps\to0$ are an active area of research (early contributions are 
\cite{ChenGeorgiouPavon.16, CominettiSanMartin.94,Leonard.12, Mikami.02, Mikami.04}, current ones are \cite{AltschulerNilesWeedStromme.21, 
Berman.20, BlanchetJambulapatiKentSidford.18, CarlierPegonTamanini.22, ConfortiTamanini.19, EcksteinNutz.22, GigliTamanini.21, NennaPegon.24, NutzWiesel.21, Pal.19, Weed.18}).
The main appeal of~\eqref{eq:epsOT} in this applied context is that it can be solved efficiently by Sinkhorn's algorithm, also called iterative proportional fitting procedure or IPFP; see \cite{PeyreCuturi.19} and its numerous references.

Sinkhorn's algorithm can be stated equivalently in primal or dual terms. In the primal formulation, it is initialized  at the probability measure $\pi_{-1}\propto e^{-c/\eps}d(\mu\otimes\nu)$ and the iterates are defined for $t\geq0$ by
\begin{align}\label{eq:SinkPrimal}
    \pi_{2t} := \argmin_{\Pi(\ast,\nu)} H(\,\cdot\,|\pi_{2t-1}), \qquad
    \pi_{2t+1} := \argmin_{\Pi(\mu,\ast)} H(\,\cdot\,|\pi_{2t}), 
\end{align}
where $\Pi(\ast,\nu)$ is the set of measures on $\X\times\Y$ with second marginal~$\nu$ (and arbitrary first marginal), and $\Pi(\mu,\ast)$ is defined analogously. %
The dual formulation produces two sequences of functions $\varphi_{t}: \X\to\R$ and $\psi_{t}: \Y\to\R$. Namely, we initialize at $\varphi_{0}:=0$  and define for $t\geq0$ the iterates
\begin{align}\label{eq:SinkDual}
\begin{split}
    \psi_{t}(y) &:= -\log \int_{\X} e^{\varphi_{t}(x)-c(x,y)/\eps} \, \mu(dx),\\
    \varphi_{t+1}(x)&:= -\log \int_{\Y} e^{\psi_{t}(y)-c(x,y)/\eps}\, \nu(dy).
\end{split}  
\end{align}
The primal and dual formulations are related by
$d\pi_{2t}=e^{\varphi_{t}\oplus\psi_{t}-c/\eps}\,d(\mu\otimes\nu)$ 
where $(\varphi\oplus\psi)(x,y):=\varphi(x)+\psi(y)$, and similarly for~$\pi_{2t-1}$ with~$\psi_{t-1}$. To wit, the minimization in~\eqref{eq:SinkPrimal} is solved in closed form by~\eqref{eq:SinkDual}. The latter boils down to a simple matrix-vector multiplication in a discretized setting, whence the suitability for high-dimensional problems.

Our main result is the convergence of the algorithm in the sense of relative entropy and its rate, under general conditions including unbounded costs. Methodologically, we proceed in two steps. First, we show that certain (exponential) moment estimates for the iterates $\varphi_{t},\psi_{t}$ imply a convergence rate. Second, we show how to obtain the necessary  estimates for important classes of cost functions~$c$ and marginals~$(\mu,\nu)$.
The following illustrates the main findings by giving a simplified statement for powers of the distance cost, covering in particular quadratic cost with subgaussian marginals $(\mu,\nu)$. We denote by $\pi_{*}\in\Pi(\mu,\nu)$ the unique solution  of~\eqref{eq:epsOT} and by $(\varphi_{*},\psi_{*})$ the associated dual potentials; see \cref{se:backgr} for the pertinent definitions.

\begin{example}
  Let $\X=\Y=\R^{d}$ and $c(x,y)=|x-y|^{p}$ where $p\in [0,\infty)$. Suppose that 
   $$
    \int e^{\lambda |x|^{p}}\, \mu(dx) +\int e^{\lambda |y|^{p}}\, \nu(dy)<\infty \quad \mbox{for some }\lambda>0.
   $$
   Let~$\pi_{t}$ and~$(\varphi_{t},\psi_{t})$ be the Sinkhorn iterates~\eqref{eq:SinkPrimal}--\eqref{eq:SinkDual}, and let~$(\mu_{t},\nu_{t})$ denote the marginals of~$\pi_{t}$. We have
   \begin{gather}
     H(\pi_{*}|\pi_{t}) + H(\pi_{t}|\pi_{*}) = O(t^{-1}), \label{eq:exPiEntrop}\\
     \int (\varphi_{*}-\varphi_{t}) \,d\mu + \int (\psi_{*}-\psi_{t}) \,d\nu = O(t^{-1}), \label{eq:exDualSubopt}\\
     H(\mu_{t}|\mu) + H(\mu|\mu_{t}) + H(\nu_{t}|\nu) + H(\nu|\nu_{t})= O(t^{-2}). \label{eq:exMarginalEntrop}
   \end{gather} 
   The constants implicit in $O(\cdot)$ above are detailed in Sections~\ref{se:Sink} and~\ref{se:biconjugates}; their dependence on the regularization parameter~$\eps$ is a scaling by~$\eps^{-2p}$.
\end{example} 

Note that convergence rates for total variation follow immediately from the above rates via Pinsker's inequality, with constants scaling like~$\eps^{-p}$.
More generally, we also cover costs $c(x,y)=d(x,y)^{p}$ for an arbitrary (measurable) distance~$d$ on a Polish space and differentiable costs~$c$ with $|D c(x,y)|\leq C(1+|x|+|y|)^{p-1}$ on a Banach space, or any measurable costs with similar growth properties. Indeed, our rates depend only on an a priori estimate of the form
\begin{align}\label{eq:C1intro}
    C_{1}&:=\sup_{t\geq0}\inf_{\alpha>0} \frac{2}{\alpha} \left(\frac32 + \log \int e^{\alpha |\varphi_{t}-\varphi_{*}|}\,d\mu\right)<\infty.
\end{align} 
Our key result (\cref{th:SinkhornRate}) is a \emph{non-asymptotic} bound of the form
  \begin{align*}
    H(\pi_{*}|\pi_{2t}) 
    &\leq \frac{C}{t-t_{1}}, \quad t>t_{1}
  \end{align*}
where $t_{1}\in\N$ is bounded explicitly and~$C$ depends at most quadratically on~$\eps^{-p}$. By a bootstrap-type argument, this turns out to imply similar non-asymptotic bounds for the other quantities in \eqref{eq:exPiEntrop}--\eqref{eq:exMarginalEntrop}; cf.\ Corollaries~\ref{co:reverseEntropyRate} and~\ref{co:suboptimality}.

To provide~\eqref{eq:C1intro}, we establish that $|\varphi_{t}(x)|\leq C(1+|x|^{p})$ for a general class of costs whenever $\int e^{\lambda |y|^p} \nu(dy)<\infty$ for some $\lambda>0$. This implies~\eqref{eq:C1intro} under the symmetric condition $\int e^{\lambda |x|^p} \mu(dx)<\infty$. Moreover, we obtain a linear dependence of~\eqref{eq:C1intro} on~$\eps^{-p}$. Mirroring a priori estimates for $c$-convex functions in optimal transport theory~\cite{GangboMcCann.96, Villani.09}, our approach applies to arbitrary ``biconjugate'' functions, including Sinkhorn iterates and dual potentials (cf.\ \cref{se:biconjugates}). In particular, it provides a reasonably general solution to the vexing problem of bounding potentials from below.

A key innovation for our convergence analysis is to use the ``weighted Csisz\'ar--Kullback--Pinsker'' inequality of Bolley--Villani~\cite{BolleyVillani.05}. On the strength of bounds like~\eqref{eq:C1intro}, the Bolley--Villani inequality  allows us to estimate the relative entropy between certain couplings by the (symmetric) relative entropy of their marginals, with constants depending on the moments of the potentials of the couplings. This also gives rise to a new stability result (\cref{th:stability}) for the optimal coupling $\pi_{*}$ of~\eqref{eq:epsOT} wrt.\ the marginals $(\mu,\nu)$ which is of independent interest. For the convergence results, the second key innovation uses the finer details of Sinkhorn's algorithm and its monotonicity properties: we relate the improvement $H(\pi_{*}|\pi_{2t})-H(\pi_{*}|\pi_{2t+2})$ to a marginal entropy and deduce a difference equation enabling us to analyze $H(\pi_{*}|\pi_{2t})$ through elementary arguments.

Sinkhorn's algorithm dates back as far as~\cite{DemingStephan.40}; see also \cite{Fortet.40, Leonard.19} for Fortet's (related but different) iteration. Early contributions to the analysis include \cite{IrelandKullback.68,Kullback.68,Sinkhorn.64,SinkhornKnopp.67}. Introduced by \cite{FranklinLorenz.89} in the discrete (matrix scaling) case and generalized to the continuous setting by~\cite{ChenGeorgiouPavon.16b}, a contraction argument can be used to show linear (i.e., geometric) convergence in the Hilbert--Birkhoff metric when the cost~$c$ is uniformly bounded. A different avenue, viable also in the multi-marginal problem, is taken in  \cite{Carlier.21} where the algorithm is seen as a dual block-coordinate ascent. Here, boundedness of~$c$ is responsible for the strong concavity which yields linear convergence of the iterates $\varphi_{t},\psi_{t}$ in $L^{2}$ and of their  suboptimality gap. A common feature of all previous (but see the end of this section) arguments for linear convergence is that (a) bounded cost is crucial and (b) the constants deteriorate \emph{exponentially} in the regularization parameter~$\eps$. For instance, the suboptimality gap decays like $\beta^{t}$ in \cite{Carlier.21},  where $\beta=1-e^{-24\|c\|_{\infty}/\eps}$. As~$\eps$ is taken to be small in practice, $\beta$ is extremely close to~$1$, arguably failing to explain the fast convergence observed in computational practice.

Using weak$^{*}$ compactness, \cite{Ruschendorf.95} obtained a qualitative convergence result (in relative entropy) relaxing the boundedness condition on the cost. However, several other conditions are introduced, including one (see~(B1) in~\cite{Ruschendorf.95}) that essentially forces~$c$  to be bounded from above in one variable and thus excludes quadratic cost with unbounded marginal supports. A different form of compactness, in the space of dual potentials, was used in~\cite{NutzWiesel.22} to show convergence $\pi_{t}\to\pi_{*}$ in total variation under the condition $\int e^{\lambda c}\,d(\mu\otimes\nu)<\infty$ for some~$\lambda>0$. Yet this approach does not yield convergence in relative entropy, and more importantly, does not yield a rate. Another qualitative result, on weak convergence $\pi_{t}\to\pi_{*}$, was stated in \cite{Nutz.20}  for a very general continuous cost~$c$, where it was observed that Sinkhorn convergence can be deduced from the weak stability of~\eqref{eq:epsOT} wrt.\ the marginals $(\mu,\nu)$. The latter was established in~\cite{GhosalNutzBernton.21b} using the geometric approach of~\cite{BerntonGhosalNutz.21} which avoids integrability conditions on~$c$. Again, the result is purely qualitative. 

The same link with stability was exploited in~\cite{EcksteinNutz.21} whose main result is the quantitative stability of the optimal coupling $\pi_{*}$ wrt.\ the marginals $(\mu,\nu)$ in the $p$-Wasserstein distance $\cW_{p}$, under certain continuity and integrability conditions on the cost and marginals. By a general result of~\cite{Leger.21}, the marginals $(\mu_{t},\nu_{t})$ of the Sinkhorn iterate~$\pi_{t}$ converge to $(\mu,\nu)$ in relative entropy at rate $O(t^{-1})$. Combining the two results, \cite{EcksteinNutz.21} shows in particular that for subgaussian marginals $(\mu,\nu)$ and costs~$c$ which are the product of two Lipschitz functions, $\cW_{2}(\pi_{t},\pi_{*})=O(t^{-1/16})$. While the approach of~\cite{EcksteinNutz.21} does not yield strong convergence (total variation or even relative entropy), it is the closest convergence result in that it covers (some) unbounded costs and is quantitative. We note that the rate in \eqref{eq:exPiEntrop} is significantly faster and that the marginal rate of~\cite{Leger.21} is improved to $O(t^{-2})$ in~\eqref{eq:exMarginalEntrop} in our context.

Stability of entropic optimal transport wrt.\ the marginals has also been studied independently of Sinkhorn's algorithm, in addition to the aforementioned works. The first stability result is due to~\cite{CarlierLaborde.20}, for a setting with bounded cost and marginals equivalent to a common reference measure with densities uniformly bounded above and below. The authors showed by a differential approach that the potentials are continuous  in~$L^{p}$ relative to the marginal densities. %
Very recently, a differential approach was also used in \cite{CarlierChizatLaborde.22} to show uniform continuity of the potentials and their derivatives in Wasserstein distance~$\cW_{2}$, in a compact (possibly multi-marginal) setting. For two marginals, uniform continuity of the potentials in~$\cW_{1}$ was previously obtained by \cite{DeligiannidisDeBortoliDoucet.21} using the Hilbert--Birkhoff projective metric. Also very recently, \cite{BayraktarEcksteinZhang.22} showed a stability result in Wasserstein distance for unbounded marginals which extends to divergences other than relative entropy. Finally, in the setting of dynamic Schr\"odinger bridges satisfying a logarithmic Sobolev inequality for the underlying dynamics and marginal distributions with finite Fisher information, \cite{ChiariniConfortiGrecoTamanini.22} used a negative order weighted homogeneous Sobolev norm to give a stability result for the relative entropy of the Schr\"odinger bridges wrt.\ their marginals. While in the dynamic setting, this result is closest to the spirit of our \cref{th:stability} as the estimate is in relative entropy and costs may be unbounded.

Since this paper was completed, several results on Sinkhorn's algorithm have been obtained, tackling some of the aforementioned issues. For quadratic cost and  unbounded continuous marginals satisfying a log-concavity condition, \cite{ConfortiDurmusGreco.23} proves linear convergence based on a fine analysis of the gradients of Schr\"odinger potentials and Sinkhorn iterates. Meanwhile, \cite{Eckstein.23} adapts Hilbert's projective metric to unbounded functions and obtains linear convergence for a general class of unbounded costs functions and sufficiently integrable marginals. For bounded and sufficiently regular cost and marginals, \cite{ChizatDelalandeVaskevicius.24} shows that linear convergence holds with a constant depending polynomially (rather than exponentially) on the regularization parameter. 

The remainder of this paper is organized as follows. \Cref{se:backgr} introduces the necessary background and notation. \Cref{se:entropyEst} provides basic estimates for the relative entropy between certain couplings in terms of their marginals, including the stability result for the optimal couplings. \Cref{se:Sink} establishes the main results on Sinkhorn's algorithm, taking for granted an a priori estimate for the dual iterates. Finally, \cref{se:biconjugates} shows how to obtain those estimates for arbitrary biconjugate functions, for a large class of costs and marginals.

\section{Background and Notation}\label{se:backgr}

In the entropic optimal transport problem~\eqref{eq:epsOT}, we may divide by~$\eps$ to reduce to the case $\eps=1$, at the expense of changing the cost function to~$c/\eps$. In what follows, we will thus assume that $\eps=1$, yet keep an eye on the scaling behavior of the results. Our problem then reads 
\begin{equation}\label{eq:EOT}
 \cC_{1}(\mu,\nu):=\inf_{\pi\in\Pi(\mu,\nu)} \int c \, d\pi +  H(\pi|\mu\otimes\nu).
\end{equation}
Throughout, $(\X,\mu)$ and $(\Y,\nu)$ are Polish probability spaces and  $c:\X\times\Y\to\R$ is a measurable function such that $e^{-c}\in L^{1}(\mu\otimes\nu)$ and~\eqref{eq:EOT} is finite. We write $\sim$ for measure-theoretic equivalence (having the same nullsets) and $a^{\pm}=\max\{\pm a,0\}$ for $a\in\R$.

The following facts can be found, e.g., in~\cite[Section~4]{Nutz.20}. The problem~\eqref{eq:EOT} admits a unique minimizer $\pi_{*}\in\Pi(\mu,\nu)$; moreover, $\pi_{*}\sim\mu\otimes\nu$ with density 
\begin{align} \label{eq:potentials}
  d\pi_{*} = e^{\varphi_{*}\oplus \psi_{*} - c}\,d(\mu\otimes\nu)
\end{align} 
for some measurable functions $\varphi_{*}:\X\to\R$ and $\psi_{*}:\Y\to\R$  called \emph{potentials}. The potentials are a.s.\ unique up an additive constant; i.e., $(\varphi_{*}-\alpha,\psi_{*}+\alpha)$ is another pair of potentials for any $\alpha\in\R$. The fact that $\pi_{*}\in\Pi(\mu,\nu)$ is equivalent to the \emph{Schr\"odinger system}
\begin{align}\label{eq:Schrodingersystem}
\begin{split}
    \psi_{*}(y) &= -\log \int e^{\varphi_{*}(x)-c(x,y)} \, \mu(dx),\\
    \varphi_{*}(x)&= -\log \int e^{\psi_{*}(y)-c(x,y)}\, \nu(dy).
\end{split}  
\end{align}

Occasionally it is convenient to absorb the cost into the ``reference'' measure $R\in\cP(\X\times\Y)$,
\begin{align}\label{eq:defR}
  dR=\xi^{-1} e^{-c}d(\mu\otimes\nu), \qquad \xi := \int e^{-c}\,d(\mu\otimes\nu)
\end{align} 
which allows us to state~\eqref{eq:EOT} as the entropy minimization
\begin{align}\label{eq:entropyMin}
    \cC_{1}(\mu,\nu) = \inf_{\pi\in\Pi(\mu,\nu)} H(\pi|R) - \log\xi.
\end{align} 
Here $\cP(\mathsf{Z})$ denotes the set of Borel probability measures on a topological space~$\mathsf{Z}$.

Recall that the primal iterates $\pi_{t}\in\cP(\X\times\Y)$ of Sinkhorn's algorithm were stated in~\eqref{eq:SinkPrimal} and the dual iterates $(\varphi_{t},\psi_{t})$ in~\eqref{eq:SinkDual}, where we now have $\eps=1$. We remark in passing that~\eqref{eq:SinkDual} can be seen as the Gauss--Seidel algorithm for the system~\eqref{eq:Schrodingersystem}; this will be relevant in \cref{se:biconjugates}. Recall also that $\varphi_{0}:=0$. With the conventions that $\pi_{-1}:=R$ and $\psi_{-1}:=-\log \xi$ as defined in~\eqref{eq:defR}, the primal and dual iterates are related for all $t\geq0$ by
\begin{align}\label{eq:primalDual}
d\pi_{2t}=e^{\varphi_{t}\oplus\psi_{t}-c}\,d(\mu\otimes\nu), \qquad d\pi_{2t-1} = e^{\varphi_{t}\oplus\psi_{t-1}-c}\,d(\mu\otimes\nu);
\end{align}
see~\cite[Algorithm~6.2]{Nutz.20}.
By construction, $\pi_{2t}$ has marginals $(\mu_{2t},\nu)$ while $\pi_{2t-1}$ has marginals $(\mu,\nu_{2t-1})$, where 
\begin{align}\label{eq:SinkMarginalDensities}
    \frac{d\mu_{2t}}{d\mu}=e^{\varphi_{t}-\varphi_{t+1}}, \qquad \frac{d\nu_{2t-1}}{d\nu}=e^{\psi_{t-1}-\psi_{t}}, \qquad t\geq0.
\end{align}
We recall from~\cite[Lemma~6.4\,(i)]{Nutz.20} that $\varphi_{t}\in L^{1}(\mu)$ and $\psi_{t}\in L^{1}(\nu)$. Moreover, writing $\mu(\varphi):=\int \varphi\,d\mu$ for brevity,
\begin{align}
   0 \leq \mu(\varphi_{t}) &\leq \mu(\varphi_{t+1}) \leq \cC_{1}(\mu,\nu) + \log \xi, && t\geq0, \label{eq:SinkMeansIncrease1}\\
   -\log \xi \leq \nu(\psi_{t}) &\leq \nu(\psi_{t+1}) \leq \cC_{1}(\mu,\nu), && t\geq-1. \label{eq:SinkMeansIncrease2}
\end{align} 
Here the first two inequalities in each line hold by~\cite[Lemma~6.4\,(iii)]{Nutz.20} while the last inequality follows from $\mu(\varphi_{t})+\nu(\psi_{t})\leq \cC_{1}(\mu,\nu)$ via the first inequality in the other line. The inequality $\mu(\varphi_{t})+\nu(\psi_{t})\leq \cC_{1}(\mu,\nu)$, in turn, holds by~\cite[Lemma~6.4\,(ii) and Proposition~6.5]{Nutz.20} or, more directly, by duality. In most applications, $c$ is nonnegative and thus $\log \xi\leq0$, so that the inequalities also hold with that term omitted.

\section{Auxiliary Entropy Estimates}\label{se:entropyEst}

In this section, we derive several auxiliary results that bound the relative entropy of couplings in terms of their marginals and potentials. For our analysis of Sinkhorn's algorithm in \cref{se:Sink}, \cref{le:stabilityEst} below will be used to estimate the relative entropy between the Sinkhorn iterate $\pi_{t}$ and the optimal coupling $\pi_{*}$ (cf.~\eqref{eq:stabilityAppliedSink}), whereas \cref{le:valueFunEstimate} will be used to bound the dual suboptimality of $\pi_{t}$ (cf.\ \cref{co:suboptimality}). As a by-product, our considerations yield a stability result for the optimal coupling with respect to the marginals, reported in \cref{th:stability}.

\begin{lemma}\label{le:BolleyVillani}
  Consider $\mu,\mu'\in\cP(\X)$ with $H(\mu'|\mu)<\infty$ and let $F:\X\to\R$ be measurable. Then
  $$
    \left |\int F \,d(\mu'-\mu) \right| \leq C(F) \left(\sqrt{H(\mu'|\mu)} + \frac12 H(\mu'|\mu)\right)
  $$
  where
  $$
    C(F):=\inf_{\alpha>0} \frac{2}{\alpha} \left(\frac32 + \log \int e^{\alpha |F|}\,d\mu\right).
  $$
\end{lemma}

\begin{proof}
  The weighted CKP inequality of Bolley--Villani \cite[Theorem~2.1]{BolleyVillani.05} states that for any measurable $\phi:\X\to[0,\infty]$,
  $$
    \|\phi \,d(\mu'-\mu)\|_{TV}\leq C_{\phi} \left(\sqrt{H(\mu'|\mu)} + \frac12 H(\mu'|\mu)\right)
  $$
  where $C_{\phi}= \frac32 + \log \int e^{2\phi}\,d\mu$. The claim follows by choosing $\phi=\frac{\alpha}{2}|F|$ and optimizing over $\alpha>0$.
\end{proof}  
 
\begin{lemma}\label{le:stabilityEst}
  Consider $\mu,\mu'\in\cP(\X)$ with $H(\mu'|\mu)<\infty$ and $\nu,\nu'\in\cP(\Y)$ with $H(\nu'|\nu)<\infty$, as well as measurable functions $f,f':\X\to\R$ and $g,g':\Y\to\R$ with
    \begin{align*}
    C_{1}&:=\inf_{\alpha>0} \frac{2}{\alpha} \left(\frac32 + \log \int e^{\alpha |f'-f|}\,d\mu\right)<\infty, \\
    C_{2}&:=\inf_{\alpha>0} \frac{2}{\alpha} \left(\frac32 + \log \int e^{\alpha |g'-g|}\,d\nu\right)<\infty.
  \end{align*} 
  Let $\pi\in\Pi(\mu,\nu)$ be of the form $d\pi=e^{f\oplus g -c}\,d(\mu\otimes\nu)$.
  \begin{itemize}
  \item[(a)] If $\pi'\in\Pi(\mu',\nu')$ and $d\pi'=e^{f'\oplus g' -c}\,d(\mu'\otimes\nu')$,
  then
  \begin{align*}
    H(\pi'|\pi) + H(\pi|\pi')
  & \leq C_{1}\sqrt{H(\mu'|\mu)} + (1+C_{1}/2) H(\mu'|\mu) + H(\mu|\mu') \\
  & \quad + C_{2} \sqrt{H(\nu'|\nu)}+ (1+C_{2}/2)H(\nu'|\nu)  + H(\nu|\nu') .
\end{align*}  
  
  \item[(b)] If $\pi'\in\Pi(\mu',\nu')$ and $d\pi'=e^{f'\oplus g' -c}\,d(\mu\otimes\nu)$,
  then
  \begin{align*}
    H(\pi'|\pi) + H(\pi|\pi')
  & \leq C_{1}\left(\sqrt{H(\mu'|\mu)}+\frac12 H(\mu'|\mu) \right) \\
  & \quad + C_{2} \left(\sqrt{H(\nu'|\nu)} + \frac12 H(\nu'|\nu) \right).
\end{align*}  
\end{itemize}
If $\nu=\nu'$, the requirement that $C_{2}<\infty$ can be dropped.
\end{lemma} 

\begin{proof}
  (a) In view of $C_{1}<\infty$ and $H(\mu'|\mu)<\infty$, we have $|f'-f|\in L^{1}(\mu')$ by the variational representation of $H(\mu'|\mu)$; cf.\ \cite[Equation~(1.4)]{Nutz.20}. Similarly, $|g'-g|\in L^{1}(\nu')$. Using the definition of relative entropy and $\pi'\in\Pi(\mu',\nu')$,
\begin{align*}
    &H(\pi'|\pi) \\
    & = \int \log \left( \frac{d\pi'}{d(\mu'\otimes\nu')} \frac{d(\mu\otimes\nu)}{d\pi} \frac{d(\mu'\otimes\nu')}{d(\mu\otimes\nu)} \right) d\pi' \\
    & =\int (f' -f) \,d\pi'+ \int (g' -g) \,d\pi' + \int \log \left(\frac{d\mu'}{d\mu}\right)\,d\pi' + \int \log \left(\frac{d\nu'}{d\nu}\right)\,d\pi'\\
    & = \int (f' -f) \,d\mu' + \int (g' -g) \,d\nu' + H(\mu'|\mu) + H(\nu'|\nu).
  \end{align*}  
  A symmetric expression holds for $H(\pi|\pi')$; note that $|f'-f|\in L^{1}(\mu)$ and $|g'-g|\in L^{1}(\nu)$ due to $C_{1}+C_{2}<\infty$ and thus $H(\mu|\mu') + H(\nu|\nu')=\infty$ if and only if $H(\pi|\pi')=\infty$. Adding the two expressions and applying \cref{le:BolleyVillani},
  \begin{align*}
    H(\pi'|\pi) + H(\pi|\pi')
  & = H(\mu'|\mu) + H(\mu|\mu') + H(\nu'|\nu)  + H(\nu|\nu') \\
  & \quad +  \int (f' -f) \,d(\mu'-\mu) + \int (g' -g) \,d(\nu'-\nu)\\
  & \leq C_{1}\sqrt{H(\mu'|\mu)}  + (1+C_{1}/2) H(\mu'|\mu) + H(\mu|\mu')\\
  & \quad + C_{2} \sqrt{H(\nu'|\nu)}+ (1+C_{2}/2)H(\nu'|\nu)  + H(\nu|\nu').
\end{align*} 
(b) Similarly as in (a),
  \begin{align*}
    H(\pi'|\pi) 
    & = \int \log \left( \frac{d\pi'}{d(\mu\otimes\nu)} \frac{d(\mu\otimes\nu)}{d\pi} \right) d\pi' \\
    & =\int (f' -f) \,d\pi'+ \int (g' -g) \,d\pi' \\
    & = \int (f' -f) \,d\mu' + \int (g' -g) \,d\nu'
  \end{align*} 
  leads to
  \begin{align*}
    H(\pi'|\pi) + H(\pi|\pi')
  & = \int (f' -f) \,d(\mu'-\mu) + \int (g' -g) \,d(\nu'-\nu),
\end{align*} 
and now the claim again follows via \cref{le:BolleyVillani}. 

Regarding the last assertion, suppose that $\nu=\nu'$. The above calculations go through (with $0* \infty:=0$) if we can ensure that $g'-g\in L^{1}(\nu)$. For a function~$G$ depending only on $y\in\Y$, clearly $G\in L^{1}(\nu)$ is equivalent to $G\in L^{1}(\pi)$ and to $G\in L^{1}(\pi')$. In general, note that if $\pi'\ll\pi$, then $\log(d\pi'/d\pi)^{-}\in L^{1}(\pi')$ as $x\mapsto x\log x$ is bounded from below. Consider first~(b). Here $\pi'\sim \pi$ and  $\log(d\pi'/d\pi) = (f' -f)\oplus (g' -g)$, and as $(f' -f)\in L^{1}(\mu+\mu')$ thanks to $C_{1}<\infty$, it follows that $(g' -g)^{-}\in L^{1}(\pi')$ and hence $(g' -g)^{-}\in L^{1}(\nu)$ as $g',g$ depend only on $y\in\Y$. Interchanging the roles of $\pi,\pi'$ yields $(g -g')^{-}\in L^{1}(\pi)$ and hence $(g' -g)^{+}\in L^{1}(\nu)$. In summary, $g'-g'\in L^{1}(\nu)$ as desired. The argument for~(a) is similar after noting that we may assume $H(\mu|\mu')<\infty$ without loss of generality.
\end{proof}

In view of~\eqref{eq:potentials}, \cref{le:stabilityEst} entails the following stability result for the optimal coupling which is of independent interest.

\begin{theorem}[Stability]\label{th:stability}
  Consider $(\mu,\nu),(\mu',\nu')\in\cP(\X)\times\cP(\Y)$ such that the associated EOT problems~\eqref{eq:EOT} are finite. Let $\pi\in\Pi(\mu,\nu)$, $\pi'\in\Pi(\mu',\nu')$ be the respective solutions and $(\varphi,\psi)$, $(\varphi',\psi')$ associated potentials. Then
  \begin{align*}
    H(\pi'|\pi) + H(\pi|\pi')
  & \leq C_{1}\sqrt{H(\mu'|\mu)} + (1+C_{1}/2) H(\mu'|\mu) + H(\mu|\mu') \\
  & \quad + C_{2} \sqrt{H(\nu'|\nu)} + (1+C_{2}/2)H(\nu'|\nu)  + H(\nu|\nu')
\end{align*}  
  where 
  \begin{align*}
    C_{1}:=\inf_{\alpha>0} \frac{2}{\alpha} \left(\frac32 + \log \int e^{\alpha |\varphi'-\varphi|}\,d\mu\right), \quad\!\!
    C_{2}:=\inf_{\alpha>0} \frac{2}{\alpha} \left(\frac32 + \log \int e^{\alpha |\psi'-\psi|}\,d\nu\right)\!.
  \end{align*} 
\end{theorem}

\begin{remark}\label{rk:multimarginal}
  \Cref{le:stabilityEst} and \cref{th:stability} generalize to the multi-marginal context where $(\mu,\nu)$, $(\mu',\nu')$ are replaced by $(\mu_{1},\dots,\mu_{N})$, $(\mu'_{1},\dots,\mu'_{N})$ with corresponding constants $C_{1},\dots,C_{N}$.
\end{remark}

\begin{remark}\label{rk:sublinearWhy}
This is an informal remark about the general approach. Neglecting higher-order terms, \cref{th:stability} is a H\"older-type estimate where the left-hand side is a relative entropy and the right-hand is the square-root of a relative entropy with some constant in front. To obtain a Lipschitz estimate along these lines, one would need to replace the constant by a term involving another square-root of relative entropy. That additional step is indeed possible in the bounded setting of~\cite{DeligiannidisDeBortoliDoucet.21}. There, it is shown that the potentials are Lipschitz continuous with respect to the marginals as a map from 1-Wasserstein space to the space of continuous functions with uniform metric, and by Pinsker's inequality, the 1-Wasserstein distance is further bounded by a square-root of relative entropy. In our more general setting, that step is not available, and we shall merely estimate the constant to be finite. At a high level, that is why we obtain H\"older stability and (in the next section) sublinear convergence, rather than Lipschitz stability and linear convergence.
\end{remark}

The following lemma is based on a similar calculation as \cref{le:stabilityEst}; it will be used to bound the dual suboptimality in \cref{co:suboptimality}. We recall the notation~\eqref{eq:defR}.
 
\begin{lemma}\label{le:valueFunEstimate}
  Consider $\mu,\mu'\in\cP(\X)$ with $H(\mu'|\mu)<\infty$ and $\nu,\nu'\in\cP(\Y)$ with $H(\nu'|\nu)<\infty$, as well as measurable functions $f,f':\X\to\R$ and $g,g':\Y\to\R$ with
    \begin{align*}
    \tilde{C}_{1}&:=\inf_{\alpha>0} \frac{2}{\alpha} \left(\frac32 + \log \int e^{\alpha |f'|}\,d\mu\right)<\infty, \\
    \tilde{C}_{2}&:=\inf_{\alpha>0} \frac{2}{\alpha} \left(\frac32 + \log \int e^{\alpha |g'|}\,d\nu\right)<\infty.
  \end{align*} 
  Let $\pi\in\Pi(\mu,\nu)$ and $\pi'\in\Pi(\mu',\nu')$ be of the form
  \begin{align*}
     d\pi=e^{f\oplus g -c}\,d(\mu\otimes\nu), \quad d\pi'=e^{f'\oplus g' -c}\,d(\mu\otimes\nu).
  \end{align*} 
  Then we have $H(\pi'|R)<\infty$ and
\begin{align*}
    H(\pi|R) - H(\pi'|R)  
  & \leq  H(\pi|\pi') + \tilde{C}_{1}\left(\sqrt{H(\mu'|\mu)}+\frac12 H(\mu'|\mu) \right) \\
  & \quad + \tilde{C}_{2} \left(\sqrt{H(\nu'|\nu)} + \frac12 H(\nu'|\nu) \right).
\end{align*} 
If $\nu=\nu'$ and $(f,g)\in L^{1}(\mu)\times L^{1}(\nu)$, the requirement that $\tilde{C}_{2}<\infty$ can be dropped.
\end{lemma} 

\begin{proof}
  In view of $H(\mu'|\mu)<\infty$ and $H(\nu'|\nu)<\infty$, the finiteness of $\tilde{C}_{1}, \tilde{C}_{2}$ implies $f'\in L^{1}(\mu+\mu')$ and $g'\in L^{1}(\nu+\nu')$, and in particular $H(\pi'|R)<\infty$. Using the definition of relative entropy, \eqref{eq:defR}, the stated form of $d\pi$ and $d\pi'$, and $\pi'\in\Pi(\mu',\nu')$,
  \begin{align*}
    H(\pi'|R) + H(\pi|\pi')
    & = \log \xi + \int f' \oplus g' \,d\pi' + \int (f-f')\oplus(g-g') \,d\pi\\
    & = \int f' \oplus g' \,d(\pi'-\pi) + \log \xi + \int f \oplus g \,d\pi\\
    & = \int f' \,d(\mu'-\mu) + \int g' \,d(\nu'-\nu) + H(\pi|R).
  \end{align*}
  We see that $H(\pi|\pi')=\infty$ if and only if $H(\pi|R)=\infty$, and in that case the claim is trivial. In the finite case, the claim follows by rearranging the above and using \cref{le:BolleyVillani} to estimate the integrals. 
  
  Turning to the last assertion, let $\nu=\nu'$ and $(f,g)\in L^{1}(\mu)\times L^{1}(\nu)$. As  $f'\in L^{1}(\mu+\mu')$ due to $\tilde{C}_{1}<\infty$, the fact that $\log(d\pi'/dR)^{-}\in L^{1}(\pi')$ implies $(g')^{-}\in L^{1}(\pi')$ and hence $(g')^{-}\in L^{1}(\nu)$. Whereas $\log(d\pi/d\pi')^{-}\in L^{1}(\pi)$ implies $(-g')^{-}\in L^{1}(\pi)$ and hence $(g')^{+}\in L^{1}(\nu)$. Thus $g'\in L^{1}(\nu)$ and now the above calculation goes through as stated.
\end{proof}

\section{Analysis of Sinkhorn's Algorithm}\label{se:Sink}

Recall that the entropic optimal transport problem~\eqref{eq:EOT} was assumed to be finite for the given marginals~$(\mu,\nu)$, that $\pi_{*}\in\Pi(\mu,\nu)$ denotes its unique optimizer and that $(\varphi_{*},\psi_{*})$ are associated potentials~\eqref{eq:potentials}. Recall also the primal Sinkhorn iterates $\pi_{t}$ defined in~\eqref{eq:SinkPrimal}, especially that $\pi_{2t}\in\Pi(\mu_{2t},\nu)$ for $t\geq0$, and the definitions of the dual iterates $(\varphi_{t},\psi_{t})$ in~\eqref{eq:SinkDual}.

The following lemma records two important monotonicity properties; the first one well known and the second due to~\cite{Leger.21}. See \cite[Proposition~6.5]{Nutz.20} and \cite[Proposition~6.10]{Nutz.20} for proofs in our setting.

\begin{lemma}\label{le:SinkMonotonicities}
  \begin{itemize}
  \item[(i)] The sequence $\{H(\pi_{*}|\pi_{t})\}_{t\geq-1}$ is decreasing.
  \item[(ii)] For all $t\geq0$,
  \begin{align*}
    &H(\mu_{2t}|\mu) \geq H(\nu|\nu_{2t+1}) \geq H(\mu_{2t+2}|\mu) \geq H(\nu|\nu_{2t+3})\geq\dots, \\
    &H(\mu|\mu_{2t}) \geq H(\nu_{2t+1}|\nu) \geq H(\mu|\mu_{2t+2})\geq H(\nu_{2t+3}|\nu)\geq\dots .
  \end{align*}
  In particular, 
  $\{H(\mu_{2t}|\mu)\}_{t\geq0}$ and $\{H(\mu|\mu_{2t})\}_{t\geq0}$ %
  are decreasing.
  \end{itemize} 
\end{lemma}

Our main results hinge on the following simple yet crucial observation.

\begin{lemma}\label{le:sinkDiff}
  For all $t\geq0$, 
  \begin{align*}
    H(\pi_{*}|\pi_{2t+2})- H(\pi_{*}|\pi_{2t})
    & = - [H(\mu|\mu_{2t}) + H(\nu|\nu_{2t+1})] \\
    & \leq - [H(\mu_{2t+2}|\mu) + H(\mu|\mu_{2t+2})]\\
    & \leq - H(\mu_{2t+2}|\mu).
  \end{align*} 
\end{lemma} 

\begin{proof}
  Using the definitions of $\pi_{2t+2}$, $\pi_{2t}$ followed by $\pi^{*}\in\Pi(\mu,\nu)$ and~\eqref{eq:SinkMarginalDensities},
  \begin{align*}
    -H(\pi_{*}&|\pi_{2t+2})+H(\pi_{*}|\pi_{2t}) \\
    & =  \int (\varphi_{t+1}\oplus\psi_{t+1}) - (\varphi_{*}\oplus\psi_{*}) \,d\pi_{*}
       + \int (\varphi_{*}\oplus\psi_{*}) -(\varphi_{t}\oplus\psi_{t})  \,d\pi_{*} \\
    & = \mu(\varphi_{t+1}-\varphi_{t}) + \nu(\psi_{t+1}-\psi_{t}) 
     = H(\mu|\mu_{2t}) + H(\nu|\nu_{2t+1}),
  \end{align*}   
  which is the claimed identity.
  By \cref{le:SinkMonotonicities} we have $H(\mu|\mu_{2t}) \geq H(\mu|\mu_{2t+2})$ and $H(\nu|\nu_{2t+1}) \geq H(\mu_{2t+2}|\mu)$, showing the first inequality. The second inequality is trivial.
\end{proof}

While we have not addressed the convergence of $H(\pi_{*}|\pi_{2t})$ yet, \cref{le:sinkDiff} entails that \emph{the convergence rate of the marginals is at least one order faster.} This remarkable fact holds in full generality, without the moment conditions to be imposed below for our main results.

\begin{proposition}\label{pr:marginalRate}
  For all $t\geq0$,
  $$
    H(\mu_{2t}|\mu) + H(\mu|\mu_{2t}) \leq \frac{2H(\pi_{*}|\pi_{2\lfloor t/2 \rfloor})
}{ t }.
  $$
\end{proposition} 

\begin{proof}
  \cref{le:sinkDiff} shows that $a_{2s} := H(\mu_{2s}|\mu) + H(\mu|\mu_{2s})$ satisfies
  \begin{align*}
    a_{2s}\leq H(\pi_{*}|\pi_{2s-2})- H(\pi_{*}|\pi_{2s}), \quad s\geq1.
  \end{align*} 
  As $(a_{2s})$ is decreasing by \cref{le:SinkMonotonicities}, it follows for 
  any $m\geq0$ that
  $$
   m a_{4m} \leq \sum_{s=m+1}^{2m} a_{2s} \leq  H(\pi_{*}|\pi_{2m}) - H(\pi_{*}|\pi_{4m}) \leq H(\pi_{*}|\pi_{2m}),
  $$
  showing the claim for $t=2m$. Similarly,
  $$
   (m+1) a_{4m+2} \leq \sum_{s=m+1}^{2m+1} a_{2s} \leq H(\pi_{*}|\pi_{2m}) - H(\pi_{*}|\pi_{4m+2}) \leq H(\pi_{*}|\pi_{2m})
  $$
  implies the claim for $t=2m+1$.
\end{proof}

In view of $H(\pi_{*}|\pi_{2\lfloor t/2 \rfloor})\leq H(\pi_{*}|R)<\infty$,  \cref{pr:marginalRate} already includes that  the marginal entropies converge at least with rate~$O(t^{-1})$. This will be improved to~$O(t^{-2})$ in \cref{co:reverseEntropyRate} below.

From now on, suppose that
\begin{align}\label{eq:C1inMainText}
    C_{1}&:=\sup_{t\geq0}\inf_{\alpha>0} \frac{2}{\alpha} \left(\frac32 + \log \int e^{\alpha |\varphi_{t}-\varphi_{*}|}\,d\mu\right)<\infty;
\end{align} 
sufficient conditions will be given in Section~\ref{se:biconjugates} (especially \cref{co:iterateBound}).
As~$\pi_{2t}$ and~$\pi_{*}$ have the common second marginal~$\nu$, \cref{le:stabilityEst} then yields
\begin{align}\label{eq:stabilityAppliedSink}
    H(\pi_{2t}|\pi_{*}) + H(\pi_{*}|\pi_{2t})
  & \leq C_{1} \left( \sqrt{H(\mu_{2t}|\mu)} + \frac12 H(\mu_{2t}|\mu)\right), \quad t\geq0.
\end{align} 
We can now state our main result. Its presentation includes a slight nuisance: to state a non-asymptotic bound, we need to account for both terms in~\eqref{eq:stabilityAppliedSink}, even though the square-root is clearly the asymptotically dominating one. This is the reason for the time~$t_{1}$ below. While the geometric convergence~(a) holds in the first few iterations before~$t_{1}$, the key result is~(b) establishing the $O(t^{-1})$ rate. See \cref{rk:HRbound} for more details on the constants and in particular their dependence on the regularization parameter~$\eps$.

\begin{theorem}\label{th:SinkhornRate}
  Let $t_{0}=\inf\{t\geq0: H(\mu_{2t}|\mu)\leq 1\}$ and $t_{1}=(t_{0}-1)\vee0$ and
  $$
    \kappa = (\tfrac32 C_{1})^{-1} \wedge (2H(\pi_{*}|R))^{-1/2}.
  $$
  Then $t_{1}$ admits the bound $t_{1}\leq ((t_{2}\wedge t_{3})-1)^{+}$, where
 \begin{align*}
   t_{2} &:= \lceil H(\pi_{*}|R) - H(\pi_{0}|R)\rceil, \\
   t_{3} &:= \inf \{t\in\N: \, \lfloor t/2 \rfloor\log(1+\kappa) + \log t \geq \log(2H(\pi_{*}|\pi_{0}))\}.
 \end{align*} 
  (a) For $0\leq t< t_{1}$,
  \begin{align*}%
    H(\pi_{*}|\pi_{2t}) 
    &\leq  H(\pi_{*}|\pi_{0}) (1+\kappa)^{-t} \\
    &\leq  H(\pi_{*}|R) (1+\kappa)^{-t}.
\end{align*} 
  (b) For $t\geq t_{1}$,
  \begin{align}
    H(\pi_{*}|\pi_{2t}) 
    &\leq \frac{1}{H(\pi_{*}|\pi_{2t_{1}})^{-1}+ \frac12 \kappa^{2}  (t-t_{1})} \nonumber\\
    &\leq \frac{1}{H(\pi_{*}|R)^{-1}+ \frac12 \kappa^{2}  (t-t_{1})} \nonumber\\
    &\leq 5 \frac{ C_{1}^{2} \vee H(\pi_{*}|R)}{t-t_{1}}. \label{eq:SinkhornRate2}
  \end{align}
\end{theorem} 

\begin{remark}[On the Constants]\label{rk:HRbound}
Recall that the general problem~\eqref{eq:epsOT} involves a regularization parameter $\eps>0$. While we have assumed $\eps=1$ in the above results, the general case is readily recovered by applying the results to the cost $c/\eps$ instead of $c$. In particular, we can  track the dependence of the constants on $\eps$ as follows.

  (a) Suppose for simplicity that $c\geq0$. Then \eqref{eq:entropyMin} yields the explicit upper bound
  \begin{align*}
    H(\pi_{*}|R)=\cC_{1}(\mu,\nu) +  \log\xi \leq \cC_{1}(\mu,\nu) \leq \int c \,d(\mu\otimes\nu)
  \end{align*} 
  as $\pi:=\mu\otimes\nu$ is an admissible control in~\eqref{eq:EOT}. Translating back to~\eqref{eq:epsOT}, we observe that this bound scales linearly in the regularization~$\eps^{-1}$.
  
  (b) When using the bounds for $\varphi_{t},\varphi_{*}$ from \cref{se:biconjugates}, the constant~$C_{1}$ of~\eqref{eq:C1inMainText} scales at most linearly in $\eps^{-p}$ for typical examples of costs; e.g.,  $c(x,y)=d(x,y)^{p}$ with $p\geq1$ (see \cref{th:unifBiconjugBound}). In particular, when~$\eps$ is small, $C_{1}$ is the dominating term in $C_{1}^{2} \vee H(\pi_{*}|R)$ and in the definition of~$\kappa$. And most importantly, \emph{the constant in~\eqref{eq:SinkhornRate2} grows at most like $\eps^{-2p}$.}
\end{remark} 

\begin{proof}[Proof of Theorem~\ref{th:SinkhornRate}]
We first prove the upper bounds for $H(\pi_{*}|\pi_{2t})$.
As $H(\mu_{2t}|\mu)$ is monotone (\cref{le:SinkMonotonicities}), we have 
$H(\mu_{2t}|\mu)\leq 1$ for $t\geq t_{0}$ and $H(\mu_{2t}|\mu)\geq 1$ for $0\leq t < t_{0}$. Consequently, \eqref{eq:stabilityAppliedSink} yields
\begin{align*}
    H(\pi_{2t}|\pi_{*}) + H(\pi_{*}|\pi_{2t}) \leq 
    \begin{cases}
      \tfrac32 C_{1} H(\mu_{2t}|\mu), & \quad 0\leq t< t_{0}, \\
      \tfrac32 C_{1}  \sqrt{H(\mu_{2t}|\mu)}, &\quad t\geq t_{0}.
    \end{cases}
\end{align*}
With $t_{1}:=(t_{0}-1)\vee0$ it follows that for any $0<\kappa\leq(\tfrac32 C_{1})^{-1}$,
\begin{align*}
    H(\mu_{2t+2}|\mu) \geq 
    \begin{cases}
       \kappa H(\pi_{*}|\pi_{2t+2}), & \quad 0\leq t< t_{1}, \\
       \kappa^{2} H(\pi_{*}|\pi_{2t+2})^{2}, &\quad t\geq t_{1}.
    \end{cases}
\end{align*}
Combining this with the last inequality in \cref{le:sinkDiff} yields
  \begin{align}\label{eq:diffIneqForSink}
    H(\pi_{*}|\pi_{2t+2})- H(\pi_{*}|\pi_{2t})
    \leq -    \begin{cases}
       \kappa H(\pi_{*}|\pi_{2t+2}), & \quad 0\leq t< t_{1}, \\
       \kappa^{2} H(\pi_{*}|\pi_{2t+2})^{2}, &\quad t\geq t_{1}
    \end{cases}
  \end{align} 
  which we can analyze through the lens of difference equations.
  
(a) Consider $0\leq t< t_{1}$ and set $F(t):=H(\pi_{*}|\pi_{2t})$. As $\kappa\leq(\tfrac32 C_{1})^{-1}$, \eqref{eq:diffIneqForSink} shows
  $$
   F(t+1)-F(t) \leq - \kappa F(t+1).
  $$ 
  We see by induction that~$F$ is dominated by the unique solution $G$ of the difference equation
  $$
    G(t+1)-G(t) = - \kappa G(t+1), \quad 0\leq t< t_{1}; \quad G(0)=F(0).
  $$
  This solution is explicitly given by
  $$
    G(t)= F(0) (1+\kappa)^{-t}, \quad 0\leq t< t_{1}.
  $$
  The second inequality follows as $F(0)= H(\pi_{*}|\pi_{0})\leq H(\pi_{*}|\pi_{-1})=H(\pi_{*}|R)$ by \cref{le:SinkMonotonicities}.  
  
(b) Let $t\geq t_{1}$. Now \eqref{eq:diffIneqForSink} reads
  $$
   F(t+1)-F(t) \leq - \kappa^{2} F(t+1)^{2}, \quad t\geq t_{1}.
  $$ 
  By an elementary analysis and induction, we see that~$F(t)$, $t\geq t_{1}$ is dominated by the unique nonnegative solution~$G$ of the difference equation
  $$
    G(t+1)-G(t) = - \kappa^{2} G(t+1)^{2}, \quad t\geq t_{1}; \quad  G(t_{1})=F(t_{1}).
  $$
  Clearly~$G$ is decreasing: $G(t+1)\leq G(t) \leq G(t_{1})=F(t_{1})$ and hence
  $$
    G(t+1) = G(t) - \kappa^{2} G(t+1)^{2}\geq G(t) - \kappa^{2} F(t_{1})G(t)=\beta G(t)
  $$
  where $\beta:=1-\kappa^{2} F(t_{1})>0$ due to 
  $\kappa^{2}\leq [2H(\pi_{*}|R)]^{-1}\leq [2F(t_{1})]^{-1}$.
   It follows that~$G$ is in turn dominated by the unique solution $S$ of
  $$
    S(t+1)-S(t) = - \gamma S(t)S(t+1), \quad t\geq t_{1}; \quad  S(t_{1})=F(t_{1})
  $$  
  with $\gamma:=\kappa^{2}\beta=\kappa^{2}(1-\kappa^{2} F(t_{1}))$, namely
  $$
    S(t)=\frac{1}{F(t_{1})^{-1}+\gamma (t-t_{1})}, \quad t\geq t_{1}.
  $$
  Using again $\kappa^{2}\leq [2F(t_{1})]^{-1}$, we have $\gamma=\kappa^{2}(1-\kappa^{2} F(t_{1}))\geq \kappa^{2}/2$, yielding the first claim. The other inequalities follow with $F(t_{1})^{-1}\geq H(\pi_{*}|R)^{-1} \geq0$, again by \cref{le:SinkMonotonicities}.
  
  (c) It remains to prove the bounds for $t_{1}$. With $A:=H(\pi_{*}|R) - H(\pi_{0}|R)$, it holds that $H(\mu_{2t}|\mu) \leq A/t$ for all $t\geq1$; cf.\ \cite[Corollary~6.12]{Nutz.20}. Hence $t_{0}\leq \lceil A \rceil=t_{2}$ and thus $t_{1}\leq (t_{2}-1)^{+}$.
On the other hand, suppose for contradiction that $t_{3}<t_{0}$. As $t_{3}\geq 2$, this implies $\lfloor t_{3}/2 \rfloor < t_{0}-1\leq t_{1}$. Using \cref{pr:marginalRate}, the bound $H(\pi_{*}|\pi_{2t})\leq  H(\pi_{*}|\pi_{0}) (1+\kappa)^{-t}$ for $0\leq t\leq t_{0}$ as proved in~(a), and the definition of $t_{3}$, we deduce
  $$
    H(\mu_{2t_{3}}|\mu) \leq 2t_{3}^{-1} H(\pi_{*}|\pi_{2\lfloor t_{3}/2 \rfloor}) \leq 2t_{3}^{-1} H(\pi_{*}|\pi_{0}) (1+\kappa)^{-\lfloor t_{3}/2 \rfloor}\leq 1,
  $$
  contradicting $t_{3}<t_{0}$.
\end{proof} 

While the above analysis crucially depended on the particular monotonicity properties of~$H(\pi_{*}|\pi_{2t})$, we can now deduce rates for the other expressions: the obtained rate $O(t^{-1})$ for $H(\pi_{*}|\pi_{2t})$ implies the rate $O(t^{-2})$ for the marginals entropies $H(\mu_{2t}|\mu)$ and $H(\mu|\mu_{2t})$, and through the latter, we can further deduce that the reverse entropy $H(\pi_{2t}|\pi_{*})$ admits the same rate $O(t^{-1})$ as $H(\pi_{*}|\pi_{2t})$.

\begin{corollary}\label{co:reverseEntropyRate}
  For $t\geq 2t_{1}$,
  \begin{align*}
    H(\mu_{2t}|\mu) + H(\mu|\mu_{2t})
    &\leq 10 \frac{ C_{1}^{2} \vee H(\pi_{*}|R)}{(\lfloor t/2 \rfloor-t_{1})t} = O(t^{-2}), \\
    H(\pi_{2t}|\pi_{*}) + H(\pi_{*}|\pi_{2t})
  & \leq  5 \frac{ C_{1}^{2} \vee (H(\pi_{*}|R)^{1/2}C_{1})}{\sqrt{(\lfloor t/2 \rfloor-t_{1})t}} = O(t^{-1}).
  \end{align*} 
\end{corollary}

\begin{proof}
  The first claim follows by combining \cref{pr:marginalRate} and \cref{th:SinkhornRate}, and then the second claim follows via~\eqref{eq:stabilityAppliedSink}.
\end{proof}  

We can state a similar result for $\int (\varphi_{*}\oplus\psi_{*} - \varphi_{t}\oplus\psi_{t})\, d(\mu\otimes\nu)$, measuring the ``suboptimality'' of $(\varphi_{t},\psi_{t})$ in the dual problem of~\eqref{eq:EOT}.

\begin{corollary}\label{co:suboptimality}
  Let
  \begin{align}
    \tilde{C}_{1}&:=\sup_{t\geq0}\inf_{\alpha>0} \frac{2}{\alpha} \left(\frac32 + \log \int e^{\alpha |\varphi_{t}|}\,d\mu\right)<\infty \label{eq:C1tilde}
  \end{align} 
  and $\bar{C}_{1}:=C_{1}+\tilde{C}_{1}$. 
  Then for all $t\geq 2t_{1}$,
  \begin{align*}
    0\leq \int (\varphi_{*}\oplus\psi_{*} - \varphi_{t}\oplus\psi_{t})\, d(\mu\otimes\nu)
      & \leq  5 \frac{ C_{1}\bar{C_{1}} \wedge (H(\pi_{*}|R)^{1/2}\bar{C_{1}})}{\sqrt{(\lfloor t/2 \rfloor-t_{1})t}} = O(t^{-1}).
  \end{align*}      
\end{corollary}

\begin{proof}
  Let $t\geq0$. 
  Using \eqref{eq:potentials}, \eqref{eq:primalDual} and the fact that $\varphi_{*}\oplus\psi_{*}$ is the maximizer of the dual problem (e.g., \cite{Nutz.20}), we have 
  \begin{align*}
  H(\pi_{*}|R) -H(\pi_{2t}|R) &= \int \varphi_{*}\oplus\psi_{*}\, d(\mu\otimes\nu)- \int \varphi_{t}\oplus\psi_{t}\, d(\mu\otimes\nu) \geq0.
  \end{align*} 
  On the other hand, note that $C_{1}<\infty$ implies $\varphi_{*}\in L^{1}(\mu)$,
  and then also $\psi_{*}\in L^{1}(\nu)$ as $H(\pi_{*}|R)<\infty$. \Cref{le:valueFunEstimate} and~\eqref{eq:stabilityAppliedSink} thus yield
  \begin{align*}
     H(\pi_{*}|R) - H(\pi_{2t}|R)
  &\leq H(\pi_{*}|\pi_{2t}) + \tilde{C}_{1}\left(\sqrt{H(\mu_{2t}|\mu)}+\frac12 H(\mu_{2t}|\mu) \right) \\
  &\leq \bar{C}_{1}\left(\sqrt{H(\mu_{2t}|\mu)}+\frac12 H(\mu_{2t}|\mu) \right).
\end{align*}  
The claim now follows from \cref{co:reverseEntropyRate}.
\end{proof}

Recall that the potentials $(\varphi_{*},\psi_{*})$ are unique only up to an additive constant. 
To state any result on the convergence $\varphi_{t}\to\varphi_{*}$, we need to choose a particular constant. Indeed, by the monotonicity~\eqref{eq:SinkMeansIncrease1}, the limit $m:=\lim_{t\to\infty}\mu(\varphi_{t})$ exists, and we choose the constant such that $\mu(\varphi_{*})=m$. Then, the following holds.

\begin{corollary}\label{co:SinkPhiConvergence}
   We have $\varphi_{t}\to\varphi_{*}$ in $L^{p}(\mu)$ for every $p\in[1,\infty)$.
\end{corollary} 

\begin{proof}
  We have $H(\pi_{*}|\pi_{2t})\to0$ by \cref{{th:SinkhornRate}} and hence $\|\pi_{2t}-\pi_{*}\|_{TV}\to0$ by Pinsker's inequality. The latter is equivalent to the convergence of the densities, $e^{\varphi_{t}\oplus\psi_{t}-c}\to e^{\varphi_{*}\oplus\psi_{*}-c}$ in $L^{1}(\mu\otimes\nu)$. As a consequence, $\varphi_{t}\oplus\psi_{t} \to \varphi_{*}\oplus\psi_{*}$ in measure~$\mu\otimes\nu$. After passing to a subsequence, we deduce that $\varphi_{t}\oplus\psi_{t} \to \varphi_{*}\oplus\psi_{*}$ pointwise on a Borel set $A\subset\X\times\Y$ with $(\mu\otimes\nu)(A)=1$. By Fubini's theorem, $\mu\{x\in\X:(x,y_{0})\in A\}=1$ for $\nu$-a.a.\ $y_{0}\in\Y$. Fix any such~$y_{0}$, then we deduce $\varphi_{t}(x)+\psi_{t}(y_{0}) \to \varphi_{*}(x)+ \psi_{*}(y_{0})$ for $\mu$-a.a.\ $x\in\X$. To wit, there are finite constants $a_{t}$ such that $\varphi_{t} +a_{t} \to \varphi_{*}$ $\mu$-a.s. In view of the uniform integrability implied by $C_{1}<\infty$, this convergence also holds in $L^{p}(\mu)$ for any $p\in[1,\infty)$. In particular, $\mu(\varphi_{t}) +a_{t} \to \mu(\varphi_{*})$. Recalling that $\varphi_{*}$ was chosen such that $\mu(\varphi_{*})=\lim_{t} \mu(\varphi_{t})$, we see that $\lim_{t} a_{t} =0$ and $\varphi_{t} \to \varphi_{*}$ in $L^{p}(\mu)$ for every~$p$. As $\varphi_{*}$ was fixed from the beginning, this must hold for the original sequence $(\varphi_{t})_{t\geq0}$.
\end{proof}

\section{Estimates for Conjugate Functions}\label{se:biconjugates}

The main goal of this section is to provide uniform (in~$t$) bounds for the Sinkhorn iterates~$\varphi_{t}$, and thus enable the application of \cref{th:SinkhornRate} to a broad class of cost functions. Our bounds are based on general considerations around (bi)conjugate functions, which will simultaneously yield bounds for Sinkhorn iterates and potentials. While upper bounds are direct from the integrability of the cost function~$c(x,y)$, lower bounds use more details about the interaction between the two variables~$(x,y)$. More precisely, our results are based on its asymptotic properties, hence are robust with respect to measurable perturbations that are bounded or of lower order growth.

  A priori estimates for conjugate ($c$-convex) functions are very familiar in optimal transport (e.g., \cite{Villani.09}). In the context of entropic optimal transport, such ideas have been used in the context of bounded or Lipschitz costs (e.g., \cite{Carlier.21, DiMarinoGerolin.20}); that line of argument is generalized by \cref{le:unifCont} below but does not apply in the setting of main interest to us. For quadratic cost, \cite{MenaWeed.19} bounds the dual potentials in order to bound certain empirical processes. While the quadratic cost is special as it is separable up to an inner product, those bounds are the closest precursors that we are aware of.

Given a measurable function $f:\X\to[-\infty,\infty]$, we define its conjugate $f^{c}:\Y\to[-\infty,\infty]$ and its biconjugate $f^{cc}:\X\to[-\infty,\infty]$ by
  \begin{align*}
    f^{c}(y)&= -\log \int e^{f(x)-c(x,y)}\,\mu(dx), \quad y\in\Y, \\
    f^{cc}(x)&= - \log \int e^{f^{c}(y)-c(x,y)}\,\nu(dy), \quad x\in\X,
  \end{align*} 
provided that the integrals are well defined. We note the abuse of notation: the first conjugation involves $c(\cdot,y)$ and $\mu$ while the second involves $c(x,\cdot)$ and $\nu$.\footnote{Moreover, an abuse of terminology: to distinguish from optimal transport theory, one might want to call~$f^{c}$ a soft conjugate or soft $c$-transform, but we have opted for brevity.} 
The results of this section will be applied in two situations:
\begin{itemize}
\item
The Sinkhorn iterates satisfy $\psi_{t}=\varphi_{t}^{c}$ and $\varphi_{t+1}=\varphi_{t}^{cc}$ for $t\geq0$, by their definition~\eqref{eq:SinkDual}.
\item
 If $(\varphi_{*},\psi_{*})$ are potentials, they solve\footnote{To be precise, while the potentials are only defined up to nullsets, we choose versions satisfying~\eqref{eq:Schrodingersystem} everywhere.} the Schr\"odinger system~\eqref{eq:Schrodingersystem} which can be restated as $\psi_{*}=\varphi_{*}^{c}$ and $\varphi_{*}=\varphi_{*}^{cc}$.   
\end{itemize}
In both situations, the involved integrals are well defined. The common feature is that $\varphi_{t+1},\varphi_{*}$ are biconjugates of some function~$f$ (either~$\varphi_{t}$ or~$\varphi_{*}$). Moreover, we have a priori bounds on quantities such as $\mu(f)$ and $\nu(f^{c})$, cf.~\eqref{eq:SinkMeansIncrease1}--\eqref{eq:SinkMeansIncrease2}, whence we do not mind having such expressions in our estimates below.

Throughout this section, $(\X,d_{\X})$ and $(\Y,d_{\Y})$ are metric spaces with arbitrary but  fixed reference points $x_{0}\in\X$, $y_{0}\in\Y$. In the case of normed spaces, those are usually taken to be the origins. For brevity, we denote $|x|:=d_{\X}(x,x_{0})$ for $x\in\X$ and $|y|:=d_{\Y}(y,y_{0})$ for $y\in\Y$, even in the metric case.

Our estimates are based on the interplay of growth properties of~$c$ with integrability properties of the marginals. Our main interest is with cost functions of superlinear growth; most practical examples with (sub)linear growth can be covered by a more direct argument detailed in \cref{se:sublinGrowth} below.

To fix ideas, consider the example that $c(x,y)=|y-x|^{2}$ is the quadratic cost on $\R\times\R$ and $\nu$ is Gaussian. Then $\int e^{\lambda |y|^2} \nu(dy)$ is infinite for some $\lambda>0$ yet finite for small enough~$\lambda$. Thus, the natural growth of the cost may fail to be exponentially integrable (especially once we recall that the cost may be scaled by a large constant~$\eps^{-1}$ corresponding to the regularization parameter in~\eqref{eq:epsOT}).
In the example, $c(x,y)=y^{2}-2xy + x^{2}$, so that we can easily isolate the term~$y^{2}$ having the critical growth. In general, we will assume that $\int e^{\lambda |y|^p} \nu(dy)<\infty$ for some $\lambda>0$, and while natural costs then have growth of order~$p$, we will aim to isolate a function~$c_{2}(y)$ such that $c(x,y)-c_{2}(y)$ has growth of order~$<p$ in~$y$. The following is a representative example of what we aim to prove.

\begin{example}\label{ex:isolateTerm}
  Let $p\in[1,\infty)$ and $\int e^{\lambda|y|^p}\,\nu(dy)<\infty$ for some~$\lambda>0$. Suppose there exists a measurable function $c_{2}:\Y\to\R$ such that
    \begin{align}\label{eq:exampleRemoveLeadingOrder}
    \begin{split}
     |c(x,y)-c_{2}(y)| &\leq  C(1+|x||y|^{p-1} + |x|^{p}),\\
    |c_{2}(y)|&\leq C(1+|y|^{p})
    \end{split}
  \end{align} 
   for some $C>0$. Then, with a constant~$K$ independent of~$f$,
    \begin{align*}
    |f^{cc}(x)| &\leq |\mu(f)| + |\nu(f^{c})| + K(1+|x|^{p}).
  \end{align*} 
\end{example}

Condition~\eqref{eq:exampleRemoveLeadingOrder} is only a special case of the bounds to be handled below, but it already covers the most important examples, as shown by the next two lemmas.

\begin{lemma}\label{le:diffCostEstimates}
   Let $p\in[1,\infty)$, let $(\X,|\cdot|_{\X})$, $(\Y,|\cdot|_{\Y})$ be Banach spaces, and consider $\X\times\Y$ with a compatible norm~$|\cdot|$. Let 
   $c:\X\times\Y\to\R$ be 
   Gateaux differentiable with Gateaux derivative satisfying $|D_{u}c(z)|\leq C(1+|z|^{p-1})$ for all $u,z\in\X\times\Y$. Then for all $(x,y)\in\X\times\Y$,
  \begin{align*}
  |c(x,y)-c(0,y)| &\leq  2^{p-1}C(1+|x|_{\X}|y|_{\Y}^{p-1} + |x|_{\X}^{p}),\\
    |c(0,y)|&\leq |c(0,0)|+C(1+|y|_{\Y}^{p}).
  \end{align*} 
\end{lemma} 

\begin{proof}
  Fix $(x,y)\in\X\times\Y$ and write $x':=(x,0)$, $y':=(0,y)$. A version of the mean value theorem (cf.\ the proof of \cite[Theorem~1.8, p.\,13]{AmbrosettiProdi.95})
  shows that for some $\lambda\in[0,1]$,
    \begin{align*}
    |c(x,y)-c(0,y)| &= |c(y'+x')-c(y')| 
    \leq |D_{x'}c(y'+\lambda x')||x'|.
  \end{align*} 
  As $|y'+\lambda x'|\leq |y'| + |x'|=|y|_{\Y}+|x|_{\X}$, the assumption then yields
   \begin{align*}
    |c(x,y)-c(0,y)|
    &\leq C(1+(|y|_{\Y} + |x|_{\X})^{p-1})|x|_{\X} \\&\leq 2^{p-1}C(1+|x|_{\X}|y|_{\Y}^{p-1} + |x|_{\X}^{p}).
  \end{align*} 
  Similarly, $|c(0,y)-c(0,0)|\leq |D_{y'}c(\lambda y')||y'|\leq C(1+|y|_{\Y}^{p})$.
\end{proof}

The second example concerns the case $\X=\Y$; exponents $p\leq1$ are handled in \cref{se:sublinGrowth} below.

\begin{lemma}\label{le:costpEstimates}
  Let $p\in[1,\infty)$, let $d$ be a metric on $\X$ and $|x|:=d(x,x_{0})$ for some $x_{0}\in\X$. 
  Then for all $x,y\in\X$,
    \begin{align*}
    \big|d(x,y)^{p}-|y|^{p}\big| \leq p2^{p-1} (|x||y|^{p-1} + |x|^{p}).
  \end{align*} 
\end{lemma} 

\begin{proof}
  The mean value theorem on~$\R$ shows that for any $a,h\in\R$ with $a\geq0$ and $a+h\geq0$ there exists $\lambda\in[0,1]$  with 
  $$
    | (a+h)^{p}-a^{p} | \leq |p(a+\lambda h)^{p-1} h| \leq (C a^{p-1} +  C |h|^{p-1})|h| \leq C a^{p-1}|h| +  C |h|^{p}
  $$
  where $C=p2^{p-1}$. 
  Noting that $d(x,y)\leq |x|+|y|$, the above with $a=|y|$ and $h=|x|$ yields
  \begin{align*}
     d(x,y)^{p}-|y|^{p} \leq (|y|+|x|)^{p}-|y|^{p} \leq C |x||y|^{p-1} +  C |x|^{p},
  \end{align*} 
  which is the desired upper bound.
  For the lower bound, suppose first that $|x|\leq |y|$ and note that $d(x,y)+|x|\geq d(y,x_{0})=|y|$. Then we can apply the above with $a=|y|$ and $h=-|x|$ to find
  \begin{align*}
     d(x,y)^{p}-|y|^{p} \geq (|y|-|x|)^{p}-|y|^{p} \geq -C |x||y|^{p-1} - C |x|^{p}.
  \end{align*} 
  Whereas if $|x|\geq |y|$, then trivially
  $d(x,y)^{p}-|y|^{p} \geq -|y|^{p} \geq - |x|^{p}$.
\end{proof}

Having justified a condition like~\eqref{eq:exampleRemoveLeadingOrder}, let us now focus on how to obtain the bounds for $f^{cc}$. We first record an auxiliary result on the equivalence between certain exponential moments and a bound on the moment-generating function, extending a familiar relationship for subgaussian measures (corresponding to $p=2$ and $q=1$).

\begin{lemma}\label{le:ExpMomentBd}
Let $0<q<p$. The existence of $\lambda>0$ such that 
\begin{align}\label{eq:expMoment}
    K:=\int e^{\lambda |y|^p} \nu(dy)<\infty
\end{align}
is equivalent to the existence of $C_{0},C>0$ such that
\begin{align}\label{eq:MGFineq}
    \int e^{t|y|^q}\nu(dy)\leq C_{0}\,e^{Ct^{p/(p-q)}}, \quad t\geq 0.
\end{align}
Moreover, $C$ and $C_{0}$ depend only on $\lambda,K,p,q$.
\end{lemma}

\begin{proof}
Assuming~\eqref{eq:expMoment}, H\"older's inequality yields
\begin{align*}
    \int e^{t|y|^q}\nu(dy) 
     & \leq \Big(\int e^{2t|y|^q - \lambda |y|^p} \nu(dy)\Big)^{1/2}\Big(\int e^{\lambda |y|^p} \nu(dy)\Big)^{1/2}\\
    & \leq \sqrt K \Big(\int e^{2t|y|^q - \lambda |y|^p} \nu(dy)\Big)^{1/2}
\end{align*}
and calculus shows that the exponent under the integral satisfies
\begin{align*}
  \sup_{\xi \geq0} \,(2t\xi^q - \lambda \xi^p) %
  =2Ct^{\frac{p}{p-q}} \qforq C:=(2q/\lambda p)^{\frac{q}{p-q}} (1- q/p).
\end{align*} 
Setting $C_{0}=\sqrt{K}$, it follows that~\eqref{eq:MGFineq} holds. Conversely, if~\eqref{eq:MGFineq} holds and a random variable~$Y$ has distribution~$\nu$ under~$P$, then
\begin{align*}
  P\{|Y|\geq s\} 
  &= P\{e^{t|Y|^{q}} \geq e^{ts^{q}}\}
  \leq e^{-ts^{q}} E[e^{t|Y|^{q}}]\\
  &\leq e^{-ts^{q}}  C_{0} e^{Ct^{p/(p-q)}}
  = C_{0} e^{Ct^{p/(p-q)}-ts^{q}}.
\end{align*} 
Optimizing the choice of~$t$ yields
\begin{align*}
  \inf_{t\geq0} \, (Ct^{p/(p-q)}-ts^{q}) = -C's^{p} \qforq C':=(p/q)((p-q)/Cp)^{p/q-1}.
\end{align*}
Thus $P\{|Y|\geq s\}\leq C_{0} e^{-C's^{p}}$, or equivalently $P\{|Y|^{p/2}\geq u\}\leq C_{0} e^{-C'u^{2}}$. That is, $Y':=|Y|^{p/2}$ is subgaussian, which is equivalent to~\eqref{eq:expMoment}; see for instance \cite[Proposition~2.5.2, p.\,22]{Vershynin.18}.
\end{proof}

We can now state the main tool. As we have opted for a symmetric presentation, its condition involves functions $c_{1}, a_{\pm}$ that are partially redundant. The main idea is, still, to isolate a well-chosen function~$c_{2}(y)$ from the cost such as to improve the integrability.

\begin{lemma}\label{le:potentialGrowthp}
  Fix $p>0$ and suppose that
  \begin{align}\label{eq:not_subgauss}
    \int e^{\lambda|y|^p}\,\nu(dy)<\infty
  \end{align} 
  for some $\lambda >0$. Fix $N\geq0$ and let $\alpha_{i}\in [0,p]$, $\beta_{i}\in[0,p)$, $1\leq i\leq N$ be such that $\alpha_{i}+\beta_{i}\leq p$. Set $\tilde{\alpha}_{i}:=p-\beta_{i} \geq \alpha_{i}$. On the strength of \cref{le:ExpMomentBd}, there are $C,C_{0}>0$ with
  \begin{align}\label{eq:MGFineqAlpha}
    \int e^{t|y|^{\beta_{i}}}\nu(dy)\leq C_{0}\,e^{Ct^{p/\tilde{\alpha}_{i}}}, \quad t\geq 0
  \end{align}
  for all $1\leq i\leq N$. Let $c$ be of the form
  $$
   c(x,y)=c_{1}(x) + c_{2}(y) + \hat c(x,y)
  $$
  where $c_{1}\in L^{1}(\mu)$, $c_{2}\in L^{1}(\nu)$ and 
  \begin{align}
    \hat c(x,y) &\leq a_{+}(x) + \sum_{i=1}^{N} K^{i}_{+}|x|^{\alpha_{i}}|y|^{\beta_{i}}, \label{eq:cbound1}\\
    \hat c(x,y) &\geq -a_{-}(x) - \sum_{i=1}^{N} K^{i}_{-}|x|^{\alpha_{i}}|y|^{\beta_{i}} \label{eq:cbound2}
  \end{align}
  with $K^{i}_{\pm}\in\R_{+}$ and $a_{\pm}$ nonnegative measurable functions such that 
  \begin{align}
    A_{+} &= \int a_{+}(x)\,\mu(dx) <\infty. \label{eq:Aplus}
  \end{align}
  Finally, set $A_{\alpha_{i}}:= \int |x|^{\alpha_{i}}\,\mu(dx)$.
  Let $f\in L^{1}(\mu)$ be such that $f^{c},f^{cc}$ are well defined. Then
  \begin{align}
    f^{c}(y)-c_{2}(y)&\leq A_{+} -\mu(f-c_{1}) + \sum_{i=1}^{N} K^{i}_{+}A_{\alpha_{i}}|y|^{\beta_{i}}, \label{eq:potentialGrowthpUpper}\\
    f^{cc}(x)-c_{1}(x)&\geq \mu(f-c_{1}) -A_{+} - \log C_{0}-a_{-}(x) \nonumber\\
    & \quad - C\sum_{i=1}^{N}  N^{\frac{p}{\tilde{\alpha}_{i}}-1}(K^{i}_{+}A_{\alpha_{i}}+K^{i}_{-}|x|^{\alpha_{i}})^{p/\tilde{\alpha}_{i}}. \label{eq:potentialGrowthpLower}
  \end{align}
\end{lemma}

\begin{proof}
 We may assume that $c_{1}=0=c_{2}$ and $c=\hat c$. By Jensen's inequality,
  \begin{align*}
  f^{c}(y)&\leq -\log \int e^{f(x)-c(x,y)}\,\mu(dx) \\
      &\leq \int (c(x,y) - f(x))\,\mu(dx) \\
      &\leq \int \Big(a_{+}(x) + \sum_{i=1}^{N} K^{i}_{+}|x|^{\alpha_{i}}|y|^{\beta_{i}}\Big) \,\mu(dx) - \mu(f)\\
      & \leq A_{+}   + \sum_{i=1}^{N} K^{i}_{+}A_{\alpha_{i}}|y|^{\beta_{i}} - \mu(f),
\end{align*}
proving the upper bound. Using it in the definition of $f^{cc}$ then yields
\begin{align*}
 -f^{cc}(x)
 & = \log \int e^{f^{c}(y)-c(x,y)}\,\nu(dy) \\
 & \leq A_{+} - \mu(f) + a_{-}(x) \\
 &\quad + \log \int \exp\Big( \sum_{i=1}^{N}K^{i}_{+}A_{\alpha_{i}}|y|^{\beta_{i}}+ \sum_{i=1}^{N}K^{i}_{-}|x|^{\alpha_{i}}|y|^{\beta_{i}}\Big)\,\nu(dy).
 \end{align*}
We then apply H\"older's inequality and~\eqref{eq:MGFineqAlpha} to obtain
\begin{align*}
& \int e^{\sum_{i=1}^{N} (K^{i}_{+}A_{\alpha_{i}}+K^{i}_{-}|x|^{\alpha_{i}})|y|^{\beta_{i}}}\,\nu(dy) \\
&\hspace{8em}\leq \prod_{i=1}^{N} \Big( \int e^{N(K^{i}_{+}A_{\alpha_{i}}+K^{i}_{-}|x|^{\alpha_{i}})|y|^{\beta_{i}}}\,\nu(dy)\Big)^{\frac{1}{N}}  \\
&\hspace{8em} \leq C_{0} \prod_{i=1}^{N} e^{CN^{p/\tilde{\alpha}_{i}-1}(K^{i}_{+}A_{\alpha_{i}}+K^{i}_{-}|x|^{\alpha_{i}})^{p/\tilde{\alpha}_{i}}},
\end{align*}
completing the proof.
\end{proof}

\begin{remark}\label{rk:constant2}
  One can generalize \cref{le:potentialGrowthp} by allowing additional functions $b_{\pm}(y)$ in \eqref{eq:cbound1}--\eqref{eq:cbound2}. The upper bound~\eqref{eq:potentialGrowthpUpper} then has an additional term $+b_{+}(y)$. In the proof of the lower bound, an application of H\"older's inequality adds a term $-\log[(\int e^{2(b_{-}(y)+b_{+}(y))}\,\nu(dy))^{1/2}]$ in~\eqref{eq:potentialGrowthpLower} and also changes~$N$ to $2N$. We have chosen the simpler statement above as the examples of our main interest do not necessitate these additional functions.
\end{remark}

The next theorem summarizes our bounds for biconjugate functions.

\begin{theorem}\label{th:unifBiconjugBound}
  Let $c$ and $\nu$ satisfy~\eqref{eq:not_subgauss}--\eqref{eq:Aplus}. Suppose in addition that $c_{1},c_{2},a_{\pm}$ have growth of order at most~$p$ and that  $\int|x|^{p}\,\mu(dx)<\infty$.
  Then
  \begin{align*}
    |f^{cc}(x)| &\leq \mu(f^{-}) + \nu((f^{c})^{-})  + K(1+|x|^{p})
  \end{align*} 
  where~$K$ is independent of~$f$. If~$c$ is replaced by $\tilde c:=c/\eps$ with $\eps\in(0,1)$, we have a corresponding bound with constant 
  \begin{align}\label{eq:Kscaling}
    \tilde K:=\eps^{-\max\{p/\alpha,1\}} K, \quad \mbox{where}\quad \alpha:=\min_{i} \tilde\alpha_{i} = p-\max_{i}\beta_{i}.
  \end{align}
\end{theorem} 

\begin{proof}
  The lower bound~\eqref{eq:potentialGrowthpLower} already states
  \begin{align*}
    f^{cc}(x) -c_{1}(x) &\geq \mu(f-c_{1}) - K(1+|x|^{p}).
  \end{align*} 
  If we replace~$c$ with $\tilde c:=c/\eps$ for $\eps\in(0,1)$, the constants in~\eqref{eq:potentialGrowthpLower} change by a factor~$\eps^{-1}$ (except $C,C_{0},A_{+},A_{\alpha_{i}}$ that do not depend on the cost function). Thus by inspection of~\eqref{eq:potentialGrowthpLower}, an analogous bound holds with the constant~$\tilde K:=\eps^{-\max\{p/\alpha,1\}} K$.

  Set $B_{+} = \int b_{+}(x)\,\mu(dx)$ and $B_{\beta_{i}}= \int |y|^{\beta_{i}}\,\nu(dy)$, then the upper bound in \cref{le:potentialGrowthp} (applied to $f^{c}$ instead of $f$) translates to 
  \begin{align*}
    f^{cc}(x)-c_{1}(x)&\leq B_{+} -\nu(f^{c}-c_{2}) + a_{+}(x) + \sum_{i=1}^{N} K^{i}_{+}B_{\beta_{i}}|x|^{\alpha_{i}}.
  \end{align*}
  This can be stated as
  \begin{align*}
    f^{cc}(x) -c_{1}(x) &\leq -\nu(f^{c}-c_{2}) + K'(1+|x|^{p})
  \end{align*} 
  and here $K'$ scales linearly in~$\eps^{-1}$. As~$c_{1},c_{2}$ have growth of order at most~$p$, the theorem follows.
\end{proof}

We note that \cref{ex:isolateTerm} is a special case of \cref{th:unifBiconjugBound}; here $\alpha=1$ and $a_{\pm}(x)=C(1+|x|^{p})$ while $c_{1}=0$.

In view of the a priori bounds~\eqref{eq:SinkMeansIncrease1}--\eqref{eq:SinkMeansIncrease2}, we deduce the following.

\begin{corollary}\label{co:iterateBound}
  Under the conditions of~\cref{th:unifBiconjugBound}, the Sinkhorn iterates $(\varphi_{t})_{t\geq0}$ and the dual potential~$\varphi_{*}$ satisfy
  \begin{align*}
    |\varphi_{t}(x)| + |\varphi_{*}(x)| &\leq K(1+|x|^{p})
  \end{align*} 
  with $K$ independent of~$t$.
  In particular, if~$\int e^{\lambda|x|^p}\,\mu(dx)<\infty$ for some $\lambda>0$, then the constants~$C_{1}$ of~\eqref{eq:C1inMainText} and $\tilde{C}_{1}$ of~ \eqref{eq:C1tilde} are finite. If $c\geq0$, then~$K,C_{1},\tilde{C}_{1}$ all admit the scaling behavior~\eqref{eq:Kscaling}.\footnote{The condition $c\geq0$ guarantees that $\log \xi\leq0$ in  \eqref{eq:SinkMeansIncrease1}--\eqref{eq:SinkMeansIncrease2}. If costs can be negative, the scaling behavior of $\xi$ needs to be examined separately.} 
\end{corollary} 

\subsection{Results for Linear and Sublinear Growth}\label{se:sublinGrowth}

While the above results hold for general $p\geq 0$, an easier route is often available for $p\in [0,1]$, at least with some additional structure. 
The following is a generalization of a familiar result for Lipschitz costs. In contrast to the above, a single conjugation is sufficient to obtain upper and lower bounds in this setting. No exponential integrability is necessary, and the bounds scale linearly in the regularization~$\eps^{-1}$. We recall that $\cW_{1}$ denotes the 1-Wasserstein distance.

\begin{lemma}\label{le:unifCont}
  Let $\omega:\R_{+}\to\R_{+}$ be concave and nondecreasing.\footnote{Note that~$\omega$ need not be a modulus of continuity: it is not assumed that $\omega(0)=0$. For instance, all bounded costs satisfy the condition~\eqref{eq:unifContCond}.} Suppose 
  \begin{align}\label{eq:unifContCond}
    |c(x,y_{1})-c(x,y_{2})|\leq \omega(d_{\Y}(y_{1},y_{2})) \qforallq y_{1},y_{2}\in\Y, \quad x\in\X.
  \end{align} 
  Then $|f^{c}(y_{1})-f^{c}(y_{2})|\leq \omega(d_{\Y}(y_{1},y_{2}))$ and, with $B_{1}:=\int |y|\, \nu(dy)$,
  \begin{align*}
     \left|f^{c}(y) -\nu(f^{c})\right| \leq \omega(\cW_{1}(\delta_{y},\nu))\leq \omega(B_{1} + |y|) \leq \omega(B_{1}) + \omega(|y|).
  \end{align*} 
\end{lemma} 

\begin{proof}
  Let $y_{1},y_{2}\in\Y$. As $|c(x,y_{1})-c(x,y_{2})|\leq \omega(d_{\Y}(y_{1},y_{2}))$,
  \begin{align*}
  -f^{c}(y_{2})
  &=\log \int e^{f(x)-c(x,y_{2})}\, \mu(dx) \\
  &\leq\log \int e^{f(x)-c(x,y_{1}) + \omega(d_{\Y}(y_{1},y_{2}))}\, \mu(dx)  \\
  &= \omega(d_{\Y}(y_{1},y_{2})) - f^{c}(y_{1}).
  \end{align*}
  The first claim follows as~$y_{1}$ and~$y_{2}$ are interchangeable. Using Jensen's inequality, we deduce that
  \begin{align*}
  \left|f^{c}(y) -\int f^{c}\,d\nu\right| \leq \int |f^{c}(y)- f^{c}(y')|\,\nu(dy')
  \leq \int \omega(d_{\Y}(y,y')) \,\nu(dy')
\end{align*}
  which, by concavity and monotonicity of $\omega$, is bounded by 
  \begin{align*}
  \omega\left(\int d_{\Y}(y,y')\,\nu(dy')\right)=\omega(\cW_{1}(\delta_{y},\nu))\leq \omega(B_{1} + |y|).
\end{align*}
Note that~$\omega$ is subadditive by concavity and $\omega(0)\geq0$.
\end{proof}

\begin{corollary}\label{co:iterateBoundUnifCont}
  Let $\mu,\nu$ have finite first moments and let $c$ satisfy~\eqref{eq:unifContCond}. Then the Sinkhorn iterates $(\varphi_{t})_{t\geq0}$ and the dual potential~$\varphi_{*}$ satisfy
  \begin{align*}
    |\varphi_{t}(x)| + |\varphi_{*}(x)| &\leq  K + \omega(|x|)
  \end{align*} 
  with $K$ independent of~$t$.
  In particular, if~$\int e^{\lambda \omega(|x|)}\,\mu(dx)<\infty$ for some $\lambda>0$, then the constants~$C_{1}$ of~\eqref{eq:C1inMainText} and $\tilde{C}_{1}$ of~ \eqref{eq:C1tilde} are finite. If $c\geq0$, then~$K,\omega,C_{1},\tilde{C}_{1}$ scale at most linearly with the cost~$c$. 
\end{corollary} 

The following complements \cref{le:costpEstimates} for $p\leq1$.

\begin{example}\label{ex:distancepLess1}
  Let $d$ be a measurable metric on $\X=\Y$ and $|x|:=d(x,x_{0})$ for some $x_{0}\in\X$. Let $A_{1}:=\int |x|\, \mu(dx)<\infty$, $B_{1}:=\int |y|\, \nu(dy)<\infty$ and
  $$
  c(x,y)=d(x,y)^{p}, \quad\mbox{where}\quad p\in[0,1].
  $$
  By \cref{le:unifCont}, any (bi)conjugate function admits the modulus of continuity $\omega(s)=s^{p}$ and 
  \begin{align*}
   \left|f^{c}(y) -\nu(f^{c})\right|\leq B_{1}^{p} + |y|^{p}, \qquad \left|f^{cc}(x) -\mu(f^{cc})\right|\leq A_{1}^{p} + |x|^{p}.
  \end{align*} 
  If $\int e^{\lambda |x|^{p}}\,\mu(dx)<\infty$ for some $\lambda>0$, \cref{co:iterateBoundUnifCont} implies that~$C_{1}$ of~\eqref{eq:C1inMainText} and $\tilde{C}_{1}$ of~\eqref{eq:C1tilde} are finite with bounds growing at most linearly in the regularization~$\eps^{-1}$.
\end{example}

\paragraph{Funding, Conflicts and Data Availability.}

Promit Ghosal was supported by NSF Grant DMS-2153661. Marcel Nutz was supported by NSF Grants DMS-1812661 and DMS-2106056. The authors have no relevant financial or non-financial interests to disclose. Data and code sharing is not applicable to this article as no data or code were generated or analyzed.

\newcommand{\dummy}[1]{}

\end{document}